\documentclass[11pt, a4paper]{amsart}
\usepackage{amsmath}
\usepackage{amsthm}
\usepackage{amssymb}
\usepackage{hyperref}
\usepackage{todonotes}

\newtheorem{theorem}{Theorem}[section]

\newtheorem{lemma}{Lemma}

\numberwithin{lemma}{section}

\newtheorem{prop}[lemma]{Proposition}

\newtheorem{definition}[lemma]{Definition}

\newenvironment{potwr}[1]{{\noindent\it \underline{Proof of Theorem \ref{#1}}}:}{\qed}

\newcommand{\R}{\mathbb R}

\newcommand{\s}{\mathbb S}
\newcommand{\bo}{\mathbb B}

\newcommand{\vol}{\operatorname{vol}}

\def \eps{\varepsilon}
\def \gradient{\nabla}
\def \Lp{L_p}
\def \L2{L_2}
\def \ent {\operatorname{Ent}}

\begin{document}

\title{Sharp affine Sobolev type inequalities via the $\Lp$ Busemann-Petty centroid inequality}

\author{J. Haddad, C.H. Jim\'enez, M. Montenegro}

\email{julianhaddad@ufmg.br}
\email{hugojimenez@mat.puc-rio.br}
\email{montene@mat.ufmg.br}

\address{Universidade Federal de Minas Gerais}
\address{Pontif\'icia Universidade Cat\'olica do Rio de Janeiro}
\address{Universidade Federal de Minas Gerais}

\maketitle

\begin{abstract}
We show that the $\Lp$ Busemann-Petty centroid inequality provides an elementary and powerful tool to the study of some sharp affine functional inequalities with a geometric content, like log-Sobolev, Sobolev and Gagliardo-Nirenberg inequalities. Our approach allows also to characterize directly the corresponding equality cases.
\end{abstract}

\section{Introduction and previous results}
\label{sec:intro}

The goal of the present work is to discuss a new approach for the study of well-known affine Sobolev type inequalities such as log-Sobolev, Sobolev and Gagliardo-Nirenberg inequalities
with sharp constants. More precisely, we wish to

\begin{itemize}

\item[$(1)$] give a unified and elementary treatment of sharp affine log-Sobolev, Sobolev and Gagliardo-Nirenberg inequalities;

\item[$(2)$] illustrate the efficiency of the $\Lp$ Busemann-Petty centroid inequality for the study of such inequalities, and by this method reveal in a more explicit way their geometrical
nature;

\item[$(3)$] determine all cases of equality in these sharp inequalities as a by-product of our method.

\end{itemize}
Our main contribution consists in providing a new point of view for such inequalities by combining tools of Convex Geometry and Analysis. We address these inequalities in a more elementary way by avoiding the use of the solution of the $L_p$ Minkowski problem or a strengthen P\'olya-Szeg\"o principle, which were fundamental ingredients when said inequalities were first obtained.

Isoperimetric inequalities provide us with tools to compare parameters associated (in general) to convex bodies. In recent years a lot of attention has been paid to extending the classical Euclidean versions of said inequalities. Among these extensions we find the class of affine isoperimetric inequalities, that remain invariant under the action of non-degenerate (often volume preserving) linear maps.

One of the main ingredients used here is the $\Lp$ Busemann-Petty centroid inequality. This affine isoperimetric inequality, first obtained by Lutwak, Yang and Zhang \cite{L-Y-Z}, generalizes the classical Busemann-Petty centroid inequality due to Petty \cite{Petty} that compares the ratio between the volume of a convex body $K$ and that of its centroid body.
A more recent proof was obtained in \cite{C-G} using the technique of shadow systems. Affine functional inequalities, as those treated in this work, have been obtained using essentially the $\Lp$ extension of the Minkowski problem, P\'olya-Szeg\"o type results and the $\Lp$ Petty projection inequality. The strong connection between the Petty projection and the Busemann-Petty (see \cite{Lut}) inequalities makes the use of the latter in this type of approach far from surprising, however, it provides us with some additional geometric information.

In order to show the efficiency of our approach, we choose three prototypes of sharp affine functional inequalities, namely sharp affine $\Lp$ log-Sobolev, Sobolev and Gagliardo-Nirenberg inequalities. Our method extends to other situations as can be quoted in the recent works \cite{VHN}, where Nguyen obtains a P\'olya-Szeg\"o type principle by using some central ideas of this work, and \cite{DHJM}, where the authors introduce and prove the sharp affine $L^2$ Sobolev trace inequality.

Logarithmic Sobolev type inequalities (or log-Sobolev inequalities) are a powerful tool in Real
Analysis, Complex Analysis, Geometric Analysis, Convex Geometry and Probability. The pioneer work by L. Gross \cite{Gr} put
forward the equivalence between a class of sharp log-Sobolev inequalities involving probability measures and hypercontractivity of the associated heat
semigroup. For the Gaussian measure, the Gross's inequality can be placed into another form known as $\L2$ log-Sobolev inequality (or $\L2$-entropy inequality). Later, several authors established the sharp $\Lp$ version of this latter one as well as its extremal functions. For a list of references concerning the $\Lp$ log-Sobolev inequality (or $\Lp$-entropy inequality), we refer to \cite{Be}, \cite{Carlen}, \cite{DPDo}, \cite{G}, \cite{Gr}, \cite{Le} and \cite{W}. For some applications on the study of optimal estimates for solutions of certain nonlinear diffusion equations, we quote  \cite{BGL}, \cite{BGL1}, \cite{DDG}, \cite{G0}, \cite{G}, \cite{W} and references therein.

Let $n \geq 1$ and $p > 1$. The Euclidean sharp $\Lp$ log-Sobolev inequality, proved by Gross \cite{Gr} and Carlen \cite{Carlen} for $p = 2$, Del Pino and Dolbeault \cite{DPDo} for $1 < p < n$ and Gentil \cite{G} for any $p > 1$, states that for any smooth function $f$ on $\R^n$, such that $\int |f|^p dx=1$, we have

\begin{equation} \label{log}
\ent(|f|^p)=\int|f|^p\log|f|^pdx\leq \frac{n}{p}\log\left(\mathcal{L}_{n,p} \int |\gradient f|^p dx\right),
\end{equation}
where
\[\mathcal{L}_{n,p} =\frac{p}{n} \left( \frac{p - 1}{e} \right)^{p - 1}
\pi^{-\frac{p}{2}} \left( \frac{\Gamma(\frac{n}{2} +
1)}{\Gamma(\frac{n(p - 1)}{p} + 1)} \right)^{\frac{p}{n}}\]
Moreover, the inequality (\ref{log}) is sharp and equality holds if and only if for some $\sigma>0$ and $x_0 \in\R^n$
\[\quad f(x)= \pm c_\sigma e^{-\frac{1}{\sigma}|x - x_0|^{\frac{p}{p-1}}}\]
for all $x \in \R^n$, where
\[
c_\sigma = \sigma^{-\frac{n(p-1)}{p}} \left( \int_{\R^n} e^{-|x|^{\frac{p^2}{p-1}}} \right)^{-\frac{1}{p}}.
\]
More specifically, extremal functions of (\ref{log}) were characterized in \cite{Carlen} for $p = 2$, \cite{DPDo} for $1 < p < n$ and \cite{G} for $p > 1$, being the uniqueness part, up to scalings and translations, proved in \cite{DPDo} for any $p > 1$. Inequality (\ref{log}) was proved by Ledoux \cite{Le} for $p = 1$ and its extremal functions were characterized by Beckner \cite{Be}.

In order to state the sharp affine Sobolev type inequalities, for each $p > 1$, let
\[
\mathcal{E}_p(f) = c_{n,p} \left(\int_{\s^{n-1}}\|\gradient_\xi f\|_p^{-n}d\xi\right)^{-\frac 1n}
\]
with
\begin{eqnarray*}
c_{n,p} = \left(n \omega_n\right)^{\frac 1n} \left(\frac{n \omega_n \omega_{p-1}}{2 \omega_{n+p-2}}\right)^{\frac 1p},
\end{eqnarray*}
where $\gradient_\xi f(x) = \gradient f(x) \cdot \xi$ and $\omega_k$ denotes the volume of the unit ball in $\R^k$.

The sharp affine $\Lp$ log-Sobolev inequality, proved independently by Haberl, Schuster, Xiao \cite{HSX} and Zhai \cite{Zhai} for $1 < p < n$, states that

\begin{theorem}
\label{thm:AffineLogSobolev}
Let $n \geq 1$ and $p > 1$. Then for any smooth function $f$ on $\R^n$ satisfying $\int |f|^pdx=1$, we have
\begin{equation} \label{log-affine}
\ent(|f|^p)=\int|f|^p\log|f|^pdx\leq \frac{n}{p}\log\left(\mathcal{L}_{n,p} \mathcal{E}_p(f)^p\right),
\end{equation}
where $\mathcal{L}_{n,p}$ is the best log-Sobolev constant given in (\ref{log}). Moreover, equality holds if and only if for some $\sigma>0$, $x_0 \in\R^n$ and $A\in SL_n$,
\[\quad f(x)= \pm c_\sigma e^{-\frac{1}{\sigma}|A(x - x_0)|^{\frac{p}{p-1}}}\]
for all $x \in \R^n$, where $SL_n$ denotes the set of $n \times n$-matrices preserving orientation and volume and
\[
c_\sigma = \sigma^{-\frac{n(p-1)}{p}} \left( \int_{\R^n} e^{-|x|^{\frac{p^2}{p-1}}} \right)^{-\frac{1}{p}}.
\]
\end{theorem}

Theorem \ref{thm:AffineLogSobolev} has been stated in \cite{HSX} only for $1 < p < n$, although their approach, based on inequalities involving rearrangements of functions and inequality (\ref{log}), allows to extend this theorem to any $p > 1$. We present a new and arguably shorter proof for any $p > 1$ based on the $\Lp$ Busemann-Petty centroid inequality and $L^p$ log-Sobolev inequality involving arbitrary norms (see \cite{G}).

The Euclidean sharp $L_\infty$ log-Sobolev inequality, proved by Fujita in \cite{Fu}, states for any Lipschitz function $f$ on $\R^n$ satisfying $\int e^{\beta f} dx=1$ that

\begin{equation} \label{log-inft}
\ent(e^{\beta f}) \leq n \log\left(\frac{\beta k_n}{e} ||\gradient f||_\infty \right),
\end{equation}
where $\beta > 0$ is a constant and

\[
k_n = \left( \frac{1}{n! \omega_n}\right)^{\frac 1n} .
\]
Moreover, equality holds for normalized functions of the form $f(x) = c - b |x - a|$, where $a \in \R^n$, $b > 0$ and $c \in \R$.

Let
\[
\mathcal{E}_\infty(f) = c_{n,\infty} \left(\int_{\s^{n-1}}\|\gradient_\xi f\|_\infty^{-n}d\xi\right)^{-\frac 1n}
\]
with
\begin{eqnarray*}
c_{n,\infty} = \left(n \omega_n\right)^{\frac 1n} .
\end{eqnarray*}

Our method also produces a new and important $L_\infty$ affine version of the sharp inequality (\ref{log-inft}) stated as follows.

\begin{theorem}
\label{thm:AffineLogSobolev-infty}
Let $n \geq 1$ and $\beta > 0$. Then for any Lipschitz function $f$ on $\R^n$ satisfying $\int e^{\beta f} dx=1$, we have
\[\ent(e^{\beta f}) \leq n \log\left(\frac{\beta k_n}{e} \mathcal{E}_\infty(f) \right).\]
Moreover, equality holds for normalized functions of the form $f(x) = c - b |Ax - a|$, where $a \in \R^n$, $b > 0$, $c \in \R$ and $A\in SL_n$.
\end{theorem}
This theorem can be seen as a limiting case of Theorem \ref{thm:AffineLogSobolev} and its proof is based on this idea.

It is easy to see that the sharp affine $\Lp$ log-Sobolev inequality is stronger than its Euclidean counterpart for any $p>1$ since $\mathcal{E}_p(f) \leq \|\gradient f\|_p$ (see page 33 of \cite{L-Y-Z-1}). Note also that the above affine inequality is invariant under volume preserving affine transformations.

A well-known family of sharp interpolation inequalities is the Euclidean Gagliardo-Nirenberg inequality which states that for any smooth function $f$ on $\R^n$,

\begin{equation}\label{dgne}
\|f\|_r \leq \mathcal{G}_{n,p,m,r} ||\nabla f||_p^\theta \|f\|^{1-\theta}_m,
\end{equation}
where $1 \leq p < n$, $1 \leq m < r < \frac{np}{n-p}$, $\theta = \frac{np(r - m)}{r(m(p - n) + np)} \in (0,1)$ and $\mathcal{G}_{n,p,m,r}$ is the best Gagliardo-Nirenberg constant which is well defined thanks to a standard combination between the H\"{o}lder's inequality and the Sobolev inequality. Inequality (\ref{dgne}) was introduced independently by Gagliardo and Nirenberg in \cite{Ga} and \cite{Ni}. Some particular cases are quite
famous. For instance, in the limiting case $r = \frac{np}{n-p}$, (\ref{dgne}) yields the Sobolev inequality proved by
Sobolev in \cite{S} with extremal functions and best constants obtained by Federer and Fleming \cite{FF} and Maz'ya \cite{Ma} for $p=1$ and by Aubin \cite{Au1} and Talenti \cite{Ta} for $1 < p < n$. Indeed, in that latter one, $f$ is an extremal function of the sharp Sobolev inequality if and only if

\[f(x) = a \left(1 + b|x - x_0|^\frac p{p-1}\right)^{1-\frac pn}\]
for some $a \in \R$, $b > 0$ and $x_0 \in \R^n$, and the best Sobolev constant, denoted by $\mathcal{S}_{n,p}$, is given by

\begin{equation} \label{bsc}
\mathcal{S}_{n,p} = \pi^{-\frac 12} n^{-\frac 1p} {\left(\frac{p-1}{n-p}\right)^{1-\frac 1p} \left(\frac{\Gamma \left(\frac{n}{2}+1\right) \Gamma (n)}{\Gamma \left(-\frac{n}{p}+n+1\right) \Gamma \left(\frac{n}{p}\right)}\right)^{\frac{1}{n}}}
\end{equation}
Two other celebrated ones are: the Nash inequality, introduced by Nash in \cite{Na}, corresponds to $p = 2$, $m = 1$ and $r = 2$, with best constant computed by Carlen and Loss \cite{CaLo}, and the Moser inequality, introduced by Moser in \cite{Mo}, corresponds to $p = 2$, $m = 2$ and $r = \frac{2(n+2)}{n}$. Later, Del Pino and Dolbeault (\cite{DPDo1} for $p = 2$ and \cite{DPDo} for $1 < p < n$) and Cordero-Erausquin, Nazaret and Villani (\cite{CNV} for $1 < p < n$) independently obtained
extremal functions and best constants for the family of parameters $r = \alpha p$ and $m = \alpha (p-1) +1$ with $\alpha \in (1, \frac n{n-p})$. Namely, in this cases, $f$ is an extremal function if and only if

\[f(x) = a \left(1 + b|x - x_0|^\frac p{p-1}\right)^{-\frac{p-1}{m-p}}\]
for some $a \in \R$, $b > 0$ and $x_0 \in \R^n$, and so $\mathcal{G}_{n,p,\alpha (p-1) +1,\alpha p}$ can be explicitly computed and its value will be presented and used later.

A stronger version of the sharp Sobolev inequality, known as sharp affine $\Lp$ Sobolev inequality, was introduced and proved by Zhang \cite{Z} for $p=1$ and by Lutwak, Yang and Zhang \cite{L-Y-Z-1} for $1<p<n$. Recently, Zhai (Corollary 1.5 of \cite{Zhai}) proved an affine version of the sharp Gagliardo-Nirenberg inequality for the same family considered in \cite{DPDo1}, \cite{DPDo} and \cite{CNV}. An alternative proof has been carried out by Haberl, Schuster and Xiao (Corollary 12 of \cite{HSX}) by using their affine P\'{o}lya-Szeg\"{o} principle as the key tool.

The sharp affine Sobolev and Gagliardo-Nirenberg inequalities are the following

\begin{theorem}
\label{thm:AffineSobolev}
Let $n > 1$ and $1 < p < n$.
Then for any smooth function $f$ on $\R^n$, we have
\[\|f\|_{\frac{n p}{n-p} } \leq \mathcal{S}_{n,p} \mathcal{E}_p(f).\]
Moreover, equality holds if and only if
\[f(x) = a\left(1 + b|A(x - x_0)|^\frac p{p-1}\right)^{1-\frac pn}\]
for some $a \in \R$, $b > 0$, $x_0 \in \R^n$ and $A \in GL_n$, where $GL_n$ denotes the set of invertible $n \times n$-matrices.
\end{theorem}

\begin{theorem}
\label{thm:AffineGN}
Let $n>1$ and $1 < p < n$, $r = \alpha p$ and $m = \alpha (p-1) +1$. Assume $\alpha \in (1, \frac n{n-p})$.
Then for any smooth function $f$ on $\R^n$, we have
\begin{equation}
\label{eqn_AffineGN}
\|f\|_r \leq \mathcal{G}_{n,p,\alpha} \mathcal{E}_p(f)^\theta \|f\|^{1-\theta}_m,
\end{equation}
where $\mathcal{G}_{n,p,\alpha} = \mathcal{G}_{n,p,\alpha (p-1) +1,\alpha p}$. Moreover, equality holds if and only if
\[f(x) = a\left(1 + b|A(x - x_0)|^\frac p{p-1}\right)^{-\frac{p-1}{m-p}}\]
for some $a \in \R$, $b > 0$, $x_0 \in \R^n$ and $A \in GL_n$.
\end{theorem}

Theorem \ref{thm:AffineSobolev} was established in \cite{L-Y-Z-1} where the main ingredients used by the authors were the $\Lp$ Petty-Projection inequality and the solution of the $\Lp$ extension of the classical Minkowski problem. Our approach makes use of the $\Lp$ Busemann-Petty centroid inequality (see Section \ref{sec:not} for precise statements) but replaces the second one by the main result of Cordero, Nazaret and Villani \cite{CNV} on general sharp Sobolev inequalities obtained by means of a mass transportation method. The connection between the $\Lp$ extension of the Minkowski problem and sharp affine Sobolev inequalities was already studied in \cite{L-Y-Z-2}. Our approach also allows us to prove Theorem \ref{thm:AffineGN} with the same two ingredients mentioned above, whereas the known proofs presented in \cite{HSX} and \cite{Zhai} use rearrangements and the P\'olya-Szeg\"o principle.

The paper is organized as follows. In section \ref{sec:not} we fix some notations to be used in the paper and state the $\Lp$ Petty-Projection inequality. In section \ref{sec:log} we provide the proofs of Theorems \ref{thm:AffineLogSobolev} and \ref{thm:AffineGN}. Proof of Theorem \ref{thm:AffineSobolev} is carried out in the same spirit of the latter, so we omit the details. Both proofs heavily rely on a central tool involving convex bodies (Theorem \ref{thm:main_ineq}) and general sharp Sobolev and Gagliardo-Nirenberg inequalities proved in \cite{G} and \cite{CNV}. The proof of Theorem \ref{thm:main_ineq} requires three auxiliary lemmas. The statement of these inequalities involving abstract norms is included in this section as well for the sake of completeness. Finally, section \ref{sec:inft} is dedicated to the proof of Theorem \ref{thm:AffineLogSobolev-infty}.

\section{Notations and key tools}
\label{sec:not}

We recall that a convex body $K\subset\R^n$ is a convex compact subset of $\R^n$ with non-empty interior. As usual $\s^{n-1}$ and $\bo^n$ stand for the unit Euclidean sphere and ball, respectively. It is known that
\[ \vol(\bo^n) = \omega_n = \frac{\pi ^{\frac{n}{2}}}{\Gamma \left(\frac{n}{2}+1\right)}.\]

For $K\subset\mathbb R^n$ as before, its support function $h_K$ is defined as $$h_K(x)=\max\{\langle x, y\rangle\ :\ y\in K\}.$$ The support function, which describes the (signed) distances of supporting hyperplanes to the origin,  uniquely characterizes $K$. We also have the gauge $\|\cdot\|_K$ and radial $r_K(\cdot)$ functions of $K$ defined respectively as
\[\|x\|_K:=\inf\{\lambda>0\ :\  x\in \lambda K\},\quad x\in\R^n\setminus\{0\},\] \[r_K(x):=\max\{\lambda>0\ :\ \lambda x\in K\},\quad x\in\R^n\setminus\{0\}.\]
Clearly, $\|x\|_K=\frac{1}{r_K(x)}$.

For $K\subset \R^n$ we define its polar body, denoted by $K^\circ$, by
\[K^\circ:=\{x\in\R^n\ :\ \langle x,y \rangle\leq 1\quad \forall y\in K\}.\]
Evidently, $h_K=r_{K^{\circ}}$. It is also easy to see that $(\lambda K)^\circ=\frac{1}{\lambda}K^\circ$ for $\lambda>0$. A simple computation using polar coordinates shows that
\[\vol(K)=\frac{1}{n}\int_{\s^{n-1}}r_K^n(x)dx=\frac{1}{n}\int_{\s^{n-1}}\|x\|^{-n}_K dx.\]

For a given convex body $K\subset\R^n$ we find in the literature many associated bodies to it, in particular Lutwak and Zhang introduced \cite{L-Z} for a body $K$ its $\Lp$-centroid body denoted by $\Gamma_pK$. This body is defined by

\[h_{\Gamma_pK}^p(x):=\frac{1}{a_1\vol(K)}\int_{K}|\langle x,y\rangle|^p dy\quad \mbox{ for }x\in\R^n,\]
where

\[ a_1 = \frac{\omega_{n+p}}{\omega_2 \omega_n \omega_{p-1}}.\]

There are some other normalizations of the $\Lp$-centroid body in the literature, the previous one is made so that $\Gamma_p\bo^n=\bo^n$.

Inequalities (usually affine invariant) that compare the volume of a convex body $K$ and that of an associated body are common in the literature. For the specific case of $K$ and $\Gamma_pK$, Lutwak, Yang and Zhang \cite{L-Y-Z} (see also \cite{C-G} for an alternative proof) came up with what it is known as the $\Lp$ Busemann-Petty centroid inequality, namely
\[\vol(\Gamma_pK)\geq \vol(K).\]
This inequality is sharp if and only if $K$ is a $0$-symmetric ellipsoid. For a comprehensive survey on $\Lp$ Brunn-Minkowski theory and other topics within Convex Geometry we refer to \cite{Sch} and references therein.

\section{The strategy of proofs}\label{sec:log}

In this section we will illustrate how the idea behind the proof provides a direct, elementary and unified approach to the sharp affine functional inequalities stated in Theorems \ref{thm:AffineLogSobolev}, \ref{thm:AffineSobolev} and \ref{thm:AffineGN}. The general strategy for the proof of these theorems will be based on Theorem \ref{thm:main_ineq} below combined with some known results concerning each inequality that we present for completeness.

\pagebreak

\begin{center}
{\it Sharp affine $\Lp$ log-Sobolev inequalities}
\end{center}

Let $C:\R^n\rightarrow \R^+$ be a $q$-homogeneous function for $q>1$, that is, $C(x)$ satisfies
\[ C(\lambda x) = \lambda^q C(x),\quad \forall \lambda\geq 0,\quad x\in\R^n.\]
It is easy to prove that for $C$ as before, its Legendre transform, defined by
 \[C^*(x) = \sup_{y \in \R^n} \{ \langle x, y \rangle - C(y) \}\]
is $p$-homogeneous where $\frac{1}{p}+\frac{1}{q}=1$.

The starting point to prove Theorem \ref{thm:AffineLogSobolev} is the following result proved by Gentil (Theorem 1.1 in \cite{G}).

\begin{theorem}\label{thm:gentil}
Let $q>1$, $n \geq 1$, and $C$ a $q$-homogeneous, even, strictly convex function. Suppose that $\frac{1}{p}+\frac{1}{q}=1$. Then for any smooth function $f$ on $\R^n$ such that $\int |f|^pdx=1$, we have
\[\ent(|f|^p)=\int|f|^p\log|f|^pdx\leq \frac{n}{p}\log\left(\mathcal{L}_C\int C^*(\gradient f)dx\right),\]
where
\[\mathcal{L}_C=\frac{p^{p+1}}{ne^{p-1}\left(\int e^{-C(x)dx}\right)^{\frac{p}{n}}}.\]
Moreover, the above log-Sobolev inequality is sharp and equality holds if

\[\quad f(x)=a\exp\left(-bC(x-\bar{x})\right)\]
for $x \in \R^n$, where $a, b>0$ and $\bar{x}\in\R^n$ satisfy $a^{-p}=\int\exp\left(-pbC(x-\bar{x})\right))dx$.
\end{theorem}

Before stating the referred Theorem \ref{thm:main_ineq}, some preliminary results will be necessary and a little bit of notation should be introduced.

\begin{definition}
Let $f$ be a non-zero real smooth function on $\R^n$ with compact support.
We define
\[C_f^*(x) := \int_{\s^{n-1}}\|\gradient_\xi f\|_p^{-n-p}|\langle x, \xi\rangle|^{p}d\xi\]
and consider its Legendre transform $C_f$ .
\end{definition}

\begin{prop}
\label{prop:Cdef}
Let $f$ be a non-zero real smooth function on $\R^n$ with compact support.
The function $C_f^*$ given in the definition above is well defined, strictly convex, $p$-homogeneous and $C_f^*(x) = 0$ if and only if $x=0$.
\end{prop}
\begin{proof}
We only prove the good definition and strict convexity of $C_f^*$, since the other statements are evident.

First note that the equality $\|\gradient_\xi f\|_p = 0$ implies $\frac {\partial f}{\partial \xi} (x) = 0$ for every $x \in \R^n$ which contradicts the fact that $f$ is non-zero and has compact support.
Then, continuity with respect to $\xi$ gives $\|\gradient_\xi f\|_p \geq \eps > 0$ for every $\xi \in \s^{n-1}$ which proves the good definition.

For the strict convexity, take $x_1, x_2 \in \R^n$ with $x_1 \neq x_2$ and $\lambda \in (0,1)$. We have
\begin{eqnarray*}
 C_f^*\left(\lambda x_1 + (1-\lambda) x_2 \right) &=& \int_{\s^{n-1}}\|\gradient_\xi f\|_p^{-n-p}\left|\lambda {\langle x_1,  \xi\rangle} + (1-\lambda) {\langle x_2,  \xi\rangle} \right|^{p} d\xi\\
 &<& \int_{\s^{n-1}}\|\gradient_\xi f\|_p^{-n-p}\left( \lambda |\langle x_1,  \xi\rangle|^p + (1-\lambda) |\langle x_2,  \xi\rangle|^p \right) d\xi\\
 &=& \lambda C_f^*(x_1) + (1-\lambda) C_f^*(x_2) .
\end{eqnarray*}
\end{proof}

Denote by
\[L_f=\{\xi\in\R^n\ :\ \|\gradient_\xi f\|_p\leq 1\},\]
\[K_f = \{x \in \R^n\ :\ C_f(x) \leq 1 \}\, .\]
We note that $K_f$ and $L_f$ are convex bodies with gauge functions $\|x\|_{K_f}=C_f(x)^{\frac{1}{q}}$ and $\|\xi \|_{L_f} = \|\gradient_\xi f\|_p$ and radial functions $r_{K_f}(x) = C_f(x)^{-\frac{1}{q}}$ and $r_{L_f}(\xi ) = \|\gradient_\xi f\|_p^{-1}$.

For convenience, we set
\[Z_p(f)=\left(\int_{\s^{n-1}}\|\gradient_\xi f\|_p^{-n}d\xi\right)^{-\frac{1}{n}}\]
and notice we have the identities $n \vol(L_f) = Z_p(f)^{-n}$ and $\mathcal{E}_p(f) = c_{n,p} Z_p(f)$.

\begin{lemma}\label{eqn_C_grad_f}
Let $f$ be a non-zero real smooth function on $\R^n$ with compact support.
We have
\[
\int_{\R^n}C_f^*(\gradient f(x))dx = n \vol(L_f) .
\]

\end{lemma}
\begin{proof}
\begin{eqnarray*}
\int_{\R^n}C_f^*(\gradient f(x))dx &=& \int_{\R^n}\int_{\s^{n-1}}\|\gradient_\xi f\|_p^{-n-p}|\langle \gradient f(x),\xi\rangle|^{p}d\xi \; dx\\
&=&\int_{\s^{n-1}}\int_{\R^n}\|\gradient_\xi f\|_p^{-n-p}|\langle \gradient f(x),\xi\rangle|^{p}dx \; d\xi \\
&=&\int_{\s^{n-1}}\|\gradient_\xi f\|_p^{-n-p}\|\gradient_\xi f\|_p^{p}d\xi\\
&=&\int_{\s^{n-1}}\|\gradient_\xi f\|_p^{-n}d\xi = n \vol(L_f).
\end{eqnarray*}
\end{proof}

\begin{lemma}\label{eqn_C_support_centroid}
Let $f$ be a non-zero real smooth function on $\R^n$ with compact support.
We have
\[C_f^*(v) = (n+p) a_1 \vol(L_f) h_{\Gamma_pL_f}^p(v).\]
\end{lemma}
\begin{proof}
We compute
\begin{eqnarray*}
a_1\vol(L_f) h_{\Gamma_p L_f}^p(v) &=& \int_{L_f} |\langle v, x \rangle |^p dx
= \int_{\s^{n-1}} \int_0^{r_{L_f}(\xi)} r^{n-1} |\langle v, r \xi \rangle|^p dr\; d\xi\\
&=& \int_{\s^{n-1}} |\langle v, \xi \rangle|^p \int_0^{r_{L_f}(\xi)} r^{n+p-1} dr\; d\xi\\
&=& \int_{\s^{n-1}} |\langle v, \xi \rangle|^p \frac {r_{L_f}(\xi)^{n+p}}{n+p}  d\xi
= C_f^*(v) \frac 1{n+p}.
\end{eqnarray*}

\end{proof}

\begin{lemma}\label{eqn_legendre_polar_equivalence}
Let $f$ be a non-zero real smooth function on $\R^n$ with compact support.
We have
\[ K_f^\circ = \{x \in \R^n\ :\ C_f^*(x) \leq \ell_q\}, \]
where $\ell_q = q^{-\frac{p}{q}} p^{-1}$.
\end{lemma}
\begin{proof}
Firstly, note that $\ell_q = \max_{t \geq 0}\{t^{\frac{1}{q}}-t\}$.

Take $x \in K_f^\circ$. By definition of $K_f$, $\langle x, z \rangle \leq 1$ whenever $C_f(z) \leq 1$.
Since $C_f$ is $q$-homogeneous, we have for any $y \in \R^n$ that $C_f( C_f(y)^{-\frac{1}{q}} \; y ) = 1$, thus
\[\langle x, C_f(y)^{-\frac{1}{q}} \; y \rangle \leq 1\]
and
\[\langle x,  y  \rangle -C_f(y) \leq C_f(y)^{\frac{1}{q}} - C_f(y)\leq \ell_q,\]
so that $C_f^*(x) \leq \ell_q$.

Now take $x \in \R^n \setminus K_f^\circ$.
By definition of $K_f^\circ$, there exists $y \neq 0$ such that
\[\langle x, y \rangle > 1 \geq C_f(y),\]
so that
\[\langle x,  y  \rangle > C_f(y)^{\frac{1}{q}}.\]
For any $t > 0$, we have
\[\langle x, t y  \rangle > (t^q C_f(y)) ^{\frac{1}{q}}\]
and
\[\langle x, t y  \rangle - C_f(t y) > ( t^q C_f(y) )^{\frac{1}{q}} - t^q C_f(y).\]
Since $C_f(y) > 0$, we may choose $t > 0$ such that the right-hand side is maximized.
Thus, we get $C_f^*(x) > \ell_q$.
\end{proof}

We are in conditions to finally state our key theorem announced above.
\begin{theorem}
\label{thm:main_ineq}
Let $f$ be a non-zero real smooth function on $\R^n$ with compact support.
\begin{equation}
\label{eqn_Lconst}
\vol(K_f)^{-\frac pn} \int_{\R^n}C_f^*(\gradient f(x))dx \leq a_2 Z_p(f)^p,
\end{equation}
where $a_2 = \frac {\ell_q n^{\frac {n+p}n } }{ a_1 (n+p)   } $. Moreover, if equality holds in \eqref{eqn_Lconst} then there exists an invertible matrix $M$ such that
\begin{equation}
\label{eqn_extre}
\mathcal{E}_p(f)^p = \det(M)^{\frac {p - n}n} \int |\gradient g|^p dx,
\end{equation}
where $g(x) = f(M^{-1} x)$.

\end{theorem}
\begin{proof}

In view of Lemma \ref{eqn_C_grad_f} and the identity $n \vol(L_f) = Z_p(f)^{-n}$, \eqref{eqn_Lconst} becomes
\[
\vol(K_f)^{-\frac pn} \leq a_2Z_p(f)^{n+p} = a_2\left(n \vol(L_f)\right)^{-\frac{p+n}{n}}.
\]

Note that since $h_{\Gamma_p L_f}^p$ is $p$-homogeneous, we have
\[(n+p)\vol(L_f)a_1h_{\Gamma_pL_f}^p(x)\leq \ell_q\]
if and only if

\[h_{\Gamma_pL_f}^p\left(\frac{\left((n+p)\vol(L_f)a_1\right)^{\frac{1}{p}} x}{\ell_q^{\frac{1}{p}}}\right) \leq 1.\]

Putting together Lemmas \ref{eqn_C_support_centroid} and \ref{eqn_legendre_polar_equivalence}, we obtain
\begin{eqnarray*}
K_f^\circ &=& \ell_q^{\frac{1}{p}} \left((n+p)\vol(L_f)a_1\right)^{-\frac{1}{p}}\{x\in \R^n :\ h_{\Gamma_p L_f}(x)\leq 1\}\\
&=& \ell_q^{\frac{1}{p}} \left((n+p)\vol(L_f)a_1\right)^{-\frac{1}{p}}(\Gamma_pL_f)^\circ
\end{eqnarray*}
and
\[K_f=((n+p)\vol(L_f)a_1)^{\frac{1}{p}} \ell_q^{-\frac{1}{p}}\Gamma_p L_f.\]

Finally, we compute
\begin{eqnarray*}
 \vol(K_f)^{-\frac pn}
&=& \vol\left( (\vol(L_f)a_1(n+p))^{\frac{1}{p}}\ell_q^{-\frac{1}{p}}\Gamma_p L_f \right)^{-\frac{p}{n}}\\
&=& \left((\vol(L_f) a_1 (n+p))^{\frac{n}{p}}\ell_q^{-\frac{n}{p}}\vol(\Gamma_p L_f)\right)^{-\frac{p}{n}}          \\
&=& (\vol(L_f) a_1 (n+p))^{-1}\ell_q \vol(\Gamma_p L_f)^{-\frac{p}{n}}\\
&\leq& \frac {\ell_q}{ a_1 (n+p) }\vol(L_f)^{- \frac {n+p}n } = a_2 (n \vol(L_f))^{- \frac {n+p}n }.
\end{eqnarray*}
where the last inequality is a consequence of the $\Lp$ Busemann-Petty centroid inequality, proving \eqref{eqn_Lconst}.

Now assume equality holds in \eqref{eqn_Lconst}, then we have $\vol(\Gamma_p L_f) = \vol(L_f)$ which in view of the equality case for the $\Lp$ Busemann-Petty centroid inequality, implies that $L_f$ and thus $K_f$ are ellipsoids.
Consider the invertible matrix $M \in \R^{n\times n}$ such that $C_f(x) = \frac 1q |M x|^q$, then its Legendre transform is $C_f^*(v) = \frac 1p |{M^{-1}}^T v|^p$.
Using the previous formula for $C_f$, we compute $\vol(K_f) = \det(M)^{-1} \omega_n q^{\frac nq}$.
Then equality in \eqref{eqn_Lconst} reduces to
\[a_2 Z_p(f)^p = \left( \det(M)^{-1} \omega_n q^{\frac nq} \right)^{-\frac{p}{n}} \int \frac 1p |{M^{-1}}^T \gradient f(x)|^p dx\] and

\begin{eqnarray*}
\mathcal{E}_p(f)^p &=& \frac {c_{n,p}^p}{a_2}\left( \det(M)^{-1} \omega_n q^{\frac nq} \right)^{-\frac{p}{n}} \int \frac 1p |{M^{-1}}^T \gradient f(x)|^p dx\\
&=& \frac {c_{n,p}^p}{a_2}\omega_n^{-\frac pn} \ell_q \det(M)^{\frac{p}{n}-1} \int |{M^{-1}}^T \gradient f(M^{-1} y)|^p dy\\
&=&  \frac {c_{n,p}^p}{a_2}\omega_n^{-\frac pn} \ell_q \det(M)^{\frac{p-n}n } \int |\gradient g(y)|^p dy.
\end{eqnarray*}
Lastly, a simple computation shows that $\frac {c_{n,p}^p}{a_2}\omega_n^{-\frac pn} \ell_q = 1$.

\end{proof}

We turn our attention now to the proof of Theorem \ref{thm:AffineLogSobolev} and in order to do this we shall need Theorem \ref{thm:gentil} and one simple computation.

\begin{lemma}\label{eqn_exp_C}
We have
\[
\int_{\R^n} e^{-C_f(x)}dx = \Gamma\left(\frac{n}{q} + 1\right) \vol(K_f).
\]

\end{lemma}
\begin{proof}

\begin{eqnarray*}
\int_{\R^n} e^{-C_f(x)}dx&=&\int_{\s^{n-1}}\int_0^\infty r^{n-1}e^{-C_f(r\xi)}drd\xi=\int_{\s^{n-1}}\int_0^\infty r^{n-1}e^{-r^{q}C_f(\xi)}drd\xi\\
&=&\Gamma\left(\frac{n}{q}\right)q^{-1}\int_{\s^{n-1}}C_f(\xi)^{\frac{-n}{q}}d\xi\\
&=&\Gamma\left(\frac{n}{q}\right)q^{-1}\int_{\s^{n-1}}\|\xi\|_{K_f}^{-n}d\xi = \Gamma\left(\frac{n}{q}\right) \frac nq \vol(K_f)= \Gamma\left(\frac{n}{q} + 1\right) \vol(K_f).
\end{eqnarray*}

\end{proof}

\begin{potwr}{thm:AffineLogSobolev}
For the function $C_f$ already constructed, we compute the constant $\mathcal{L}_{C_f}$ in Theorem \ref{thm:gentil} as
\begin{eqnarray*}
\mathcal{L}_{C_f} &=& \frac{p^{p+1}}{ne^{p-1}\left(\int e^{-C_f(x)dx}\right)^{\frac{p}{n}}}
= \frac{p^{p+1}}{ne^{p-1}}  \Gamma\left(\frac{n}{q} + 1\right)^{-\frac pn} \vol(K_f)^{-\frac pn}.
\end{eqnarray*}
Evoking Theorem \ref{thm:main_ineq}, we get
\begin{align}
\label{form:main_ineq_1}
\mathcal{L}_{C_f} \int_{\R^n}C_f^*(\gradient f(x))dx
&\leq \frac{p^{p+1}}{ne^{p-1}}  \Gamma\left(\frac{n}{q} + 1\right)^{-\frac pn} a_2 Z_p(f)^p\\
&= \frac{p^{p+1}}{ne^{p-1}}  \Gamma\left(\frac{n}{q} + 1\right)^{-\frac pn} \frac {a_2 }{c_{n,p}^p} \mathcal{E}_p(f)^p \nonumber \\
&= \mathcal{L}_{n,p} \mathcal{E}_p(f)^p, \nonumber
\end{align}
deducing the desired inequality.

Assume for some $f$ the inequality in Theorem \ref{thm:AffineLogSobolev} is an equality, then we must have equality in \eqref{form:main_ineq_1}.
The equality case of Theorem \ref{thm:main_ineq} provides
\[\ent(|f|^p)=\int|f|^p\log|f|^pdx = \frac{n}{p}\log\left(\mathcal{L}_{n,p} \det(M)^{\frac{p-n} n} \int |\gradient g|^p dx\right),\]
where $g(x) = f(M^{-1} x)$.
We compute $\int |g|^p dx = \det(M)$ and taking $h(x) = g(x)/\det(M)^{\frac{1}{p}}$, we derive
\begin{eqnarray*}
\ent(|h|^p) &=& \ent(|f|^p) - \log(\det(M)) = \frac{n}{p}\log\left(\mathcal{L}_{n,p} \det(M)^{\frac pn - 1} \int |\gradient g|^p dx\right) - \log(\det(M))\\
&=& \frac{n}{p}\log\left(\mathcal{L}_{n,p} \int |\gradient h|^p dx\right).
\end{eqnarray*}
Finally, the characterization of the extremal functions of (\ref{log}) mentioned in the introduction completes the proof.
\end{potwr}
\\

\begin{center}
{\it Sharp affine Gagliardo-Nirenberg inequalities}
\end{center}

As a by-product we also provide a proof of the sharp affine version of the Gagliardo-Nirenberg inequality. The starting point of the Gagliardo-Nirenberg case is the following theorem proved by Cordero, Nazaret and Villani (Theorem 4 of \cite{CNV}) which we state under the form that will be used. The proof of the Sobolev case goes in a similar spirit replacing the latter by Theorem 2 of \cite{CNV}, among other natural modifications.

\begin{theorem}\label{thm:cordero_GN}
Let $q>1$, $n > 1$ and $C$ be a $q$-homogeneous, even, strictly convex function.
Suppose that $\frac{1}{p}+\frac{1}{q}=1$ and $1 < p < n$.
Let $\alpha \in (1, \frac n{n-p}]$, $r = \alpha p$, $m = \alpha(p-1)+1$.
Then for any smooth function $f$ on $\R^n$, we have
\[ \|f\|_r \leq \mathcal{G}_C \left(\int_{\R^n} C^*(\gradient f)dx \right)^{\frac \theta p} \|f\|^{1-\theta}_m .\]
Moreover, the above inequality is sharp and equality holds for the extremal function $h_{\alpha,p}(x) = (\sigma_{\alpha, p} + (\alpha-1) C(x))^{\frac 1{1-\alpha}}$, where $\sigma_{\alpha, p}>0$ is chosen so that $\|h_{\alpha, p}\|_r = 1$.
\end{theorem}

\begin{potwr}{thm:AffineGN}
Firstly, we compute the constant
\begin{eqnarray*}
\sigma_{\alpha, p} &=& \left(\frac{q (\alpha -1)^{\frac{n}{q}} \Gamma \left(\frac{p \alpha }{\alpha - 1}\right)}{n \vol(\{x \in \R^n :\ C(x) \leq 1 \}) \Gamma \left(\frac{n}{q}\right) \Gamma \left(\frac{\frac{p q \alpha }{\alpha - 1} - n}{q}\right)}\right)^{-\frac{(1-\alpha ) q}{\alpha  n-n-\alpha  p q}}
\end{eqnarray*}
and then the constant
\begin{eqnarray*}
\mathcal{G}^{-1}_C &=& \| \gradient h_{\alpha,p}\|_p^\theta \|h_{\alpha,p}\|_m^{1-\theta}\\
&=& (\alpha -1)^{-\theta } \left(\frac{n}{p}\right)^{\frac{\theta}{p}} \Gamma \left(n \left(\frac{1}{p}-1\right)+\frac{p \alpha }{\alpha -1}\right)^{-\frac{1}{\alpha  p}} \Gamma \left(\frac{p \alpha }{\alpha -1}\right)^{\theta  \left(\frac{\frac{\alpha  p}{\alpha -1}-1}{n}-1\right)} \Gamma \left(\frac{p \alpha }{\alpha -1}-1\right)^{\theta +\frac{\alpha  \theta  p}{n-\alpha  n}} \\
&\times& \Gamma \left(n \left(\frac{1}{p}-1\right)+\frac{p \alpha }{\alpha -1}-1\right)^{\theta  \left(\frac{\alpha  p}{(\alpha -1) n}+\frac{1}{p}-1\right)} \left(\vol(\{x \in \R^n :\ C(x) \leq 1 \}) \Gamma \left(-\frac{n}{p}+n+1\right)\right)^{\frac{\theta}{n}}\\
&=&c_1^\theta \vol(K_f)^{\frac{\theta}{n}},
\end{eqnarray*}
so that we obtain
\begin{eqnarray*}
\|f\|_r &\leq& \mathcal{G}_{C_f} \left( \int_{\R^n} C_f^*(\gradient f(x)) dx \right)^{\frac{\theta}{p}} \|f\|^{1-\theta}_m\\
 &=& c_1^{-\theta} \left( \vol(K_f)^{-\frac pn} \int_{\R^n} C_f^*(\gradient f(x)) dx \right)^{\frac{\theta}{p}} \|f\|^{1-\theta}_m\\
&\leq& \frac{a_2^{\frac{\theta}{p}}} {c_1^\theta} Z_p(f)^\theta \|f\|^{1-\theta}_m =  \frac{a_2^{\frac{\theta}{p}}} {c_1^\theta c_{n,p}^\theta }\mathcal{E}_p(f)^\theta \|f\|^{1-\theta}_m .
\end{eqnarray*}

Finally,
\[
\|f\|_r \leq c_2^\theta \mathcal{E}_p(f)^\theta \|f\|^{1-\theta}_m,
\]
where
\begin{eqnarray*}
c_2 &=& \frac{(\alpha -1)^2 p q^{-\frac{1}{q}} n^{-\frac{1}{p}} \left(\frac{\alpha  p}{\alpha -1}-1\right)^{\frac{\alpha  n-n-\alpha  p}{(\alpha -1) n}} \left(n \left(\frac{1}{p}-1\right)+\frac{\alpha  p}{\alpha -1}-1\right)^{\frac{\alpha  n-n+\alpha  p^2}{(\alpha -1) n p}}}{\sqrt{\pi } ((p-1) (-\alpha  n+n+\alpha  p)+p)}\\
&\times&  \left(\frac{\Gamma \left(-\frac{n}{p}+n+1\right) \Gamma \left(n \left(\frac{1}{p}-1\right)+\frac{p \alpha }{\alpha -1}\right)}{\Gamma \left(\frac{n}{2}+1\right) \Gamma \left(\frac{p \alpha }{\alpha -1}\right)}\right)^{-\frac{1}{n}} .
\end{eqnarray*}
Note that after some computations, $\mathcal{G}_{n,p,\alpha} = c_2^\theta$ reduces to the constant obtained in Corollary 1.3 of \cite{Zhai}, which is also the best Gagliardo-Nirenberg constant obtained in Theorem 1.2 of \cite{DPDo}.

The classification of the extremal functions is as follows. Let $f$ be an extremal function of \eqref{eqn_AffineGN}, then by \eqref{eqn_extre} we have
\[\|f\|_r = \mathcal{G}_{n,p,\alpha} \left( \det(M)^{\frac {p - n}{n p}} \|\gradient g\|_p \right)^\theta \|f\|^{1-\theta}_m,\]
where $g(x) = f(M^{-1} x)$. Then,
\begin{eqnarray*}
\|g\|_r &=& \left( \int \det(M) f(x)^r dx\right)^{\frac{1}{r}} \\
&=& \det(M)^{\frac 1r - \theta \frac {n-p}{n p}} \mathcal{G}_{n,p,\alpha} \|\gradient g\|_p^\theta \|f\|^{1-\theta}_m\\
&=& \det(M)^{\frac 1r - \theta \frac {n-p}{n p} - \frac {1-\theta} m} \mathcal{G}_{n,p,\alpha} \|\gradient g\|_p^\theta \|g\|^{1-\theta}_m\\
&=& \mathcal{G}_{n,p,\alpha} \|\gradient g\|_p^\theta \|g\|^{1-\theta}_m
\end{eqnarray*}
and thus $g$ must be an extremal function for the Euclidean Gagliardo-Nirenberg inequality. This completes the proof.
\end{potwr}

\section{Proof of Theorem \ref{thm:AffineLogSobolev-infty}}
\label{sec:inft}

Let $H(x)$ be an arbitrary norm on $\R^n$ and $p, q > 1$ such that $\frac{1}{p}+\frac{1}{q}=1$.  Write $C(x) = H(x)^q$ for $x \in \R^n$. Then, Lemma \ref{eqn_legendre_polar_equivalence} together with a simple computation gives
\[C^*(v) = \bar H(v)^p \ell_q,\]
where $\bar H$ is the norm of the polar body defined by $H$ (i.e. the dual norm of $H$).

Applying Theorem \ref{thm:gentil} to the function $e^{\frac{\beta}{p} f}$ and taking $p \to \infty$ we obtain a general version of the $L_\infty$ log-Sobolev inequality, proved by Fujita in \cite{Fu}. Namely,
\begin{equation} \label{1}
\ent(e^{\beta f}) \leq n \log\left( \mathcal L_{n,\infty, H} \|\bar H(\gradient f) \|_{\infty} \right),
\end{equation}
where $\mathcal L_{n,\infty, H} = \frac \beta e \left( \int_{\R^n} e^{-H(x)} dx \right)^{-\frac{1}{n}}$.

In order to turn this one into an affine $L_\infty$ log-Sobolev inequality, consider the convex body $L_f$ so that $\bar H_f = h_{L_f}$. This latter replaces Lemma \ref{eqn_C_support_centroid} in our method.

Define $K_f = \{x \in \R^n : H_f(x) \leq 1\}$, where $H_f$ denotes the dual norm of $\bar H_f$.

The equality
\[ K_f^\circ = \{\bar H_f\leq 1\}\]
is trivial and replaces Lemma \ref{eqn_legendre_polar_equivalence}. So, the relation between $K_f$ and $L_f$ becomes trivial. In fact, $K_f = L_f$.

As usual,
\[ \vol(L_f) = \frac 1n \int_{\s^{n-1}}  \|\gradient_\xi f\|_{\infty}^{-n} d\xi.\]

The computation of the right-hand side of (\ref{1}), which is done in Lemma \ref{eqn_C_grad_f}, now becomes almost trivial, once
\[\bar H_f(\gradient f(x)) = \max\{ \gradient_\xi f(x): \  \|\gradient_\xi f\|_{\infty} \leq 1\} \leq 1.\]

The computation of the constant $\mathcal L_{p,\infty, H_f}$, which replaces Lemma \ref{eqn_exp_C}, gives
\[\mathcal L_{p,\infty, H_f} = \frac \beta e \left( \vol(K_f) \Gamma(n+1) \right)^{-\frac{1}{n}}.\]

In conclusion,
\begin{align*}
\ent(e^{\beta f}) &\leq n \log\left( \frac \beta e (n^{-1} \Gamma(n+1))^{-\frac{1}{n}} \left( \int_{\s^{n-1}}  \|\gradient_\xi f\|_{\infty} d\xi \right)^{-\frac{1}{n}} \right)\\
&= n \log\left( \frac \beta e (\Gamma(n+1) \omega_n)^{-\frac{1}{n}} \mathcal E_\infty(f) \right)\\
&= n \log\left( \frac{\beta k_n}{e} \mathcal E_\infty(f) \right),
\end{align*}
so the desired inequality follows. To see its sharpness, choose $f(x) = c - b |Ax - a|$ such that $\int e^{\beta f} dx=1$, where $a \in \R^n$, $b > 0$, $c \in \R$ and $A\in SL_n$. In this case, an integration by parts gives

\[
\ent(e^{\beta f}) = \beta c - n.
\]
Moreover, the normalization $\int e^{\beta f} dx=1$ provides

\[
e^{-\beta c} = \frac{n! \omega_n}{\beta^n b^n} .
\]
So,

\[
\ent(e^{\beta f}) = n \log \left(\frac{\beta k_n}{e} b \right) .
\]
On the other hand, note that $\|\gradient_\xi f\|_\infty = b$ for every $\xi \in \s^{n-1}$, and so we have

\[
n \log\left(\frac{\beta k_n}{e} \mathcal{E}_\infty(f) \right) = n \log\left(\frac{\beta k_n}{e} b\ c_{n,\infty} (n \omega_n)^{-\frac{1}{n}} \right) = n \log \left(\frac{\beta k_n}{e} b \right).
\]
This ends the proof.\ \qed\\

\noindent {\bf Acknowledgments:}
The first author was supported by CAPES (BJT 064/2013). The second author acknowledges the support provided by CAPES and IMPA. The third author was supported by CAPES (BEX 6961/14-2), CNPq (PQ 306406/2013-6) and Fapemig (PPM 00223-13). The authors are indebted to F. Schuster for pointing out several references and useful comments on a previous version of this work and as well to the referee for several valuable comments and for proposing the study with our method of an affine $L^\infty$ version of log-Sobolev inequalities proved in the last section.

\end{document}